\documentclass{amsproc}

\newtheorem{theorem}{Theorem}[section]
\newtheorem{lemma}[theorem]{Lemma}

\theoremstyle{definition}

\theoremstyle{remark}

\numberwithin{equation}{section}

\begin{document}

\title{A note on complete hyperbolic structures on ideal triangulated 3-manifolds }

\author{Feng Luo}
\address{Department of Mathematics, Rutgers University, New Brunswick, New Jersey 08854}
\email{fluo@math.rutgers.edu }
\thanks{The work is supported in part by a NSF Grant.}


\date{Oct. 1, 2010.}

\dedicatory{Dedicated to Bus Jaco on the occasion of his 70th
birthday}


\begin{abstract}
It is a theorem of Casson and Rivin that the complete hyperbolic
metric on a cusp end ideal triangulated 3-manifold maximizes
volume in the space of all positive angle structures. We show that
the conclusion still holds if some of the tetrahedra in the
complete metric are flat.
\end{abstract}

\maketitle

\section{Introduction}
\subsection{    }
Epstein and Penner \cite{EP} proved that a non-compact finite
volume complete hyperbolic 3-manifold has a geodesic triangulation
in which each 3-simplex is a (possibly flat) ideal hyperbolic
tetrahedron. Here a flat ideal hyperbolic tetrahedron is a
tetrahedron with all dihedral angles being $0$ and $\pi$. The
purpose of this paper is to show that a geodesic ideal
triangulation of a complete hyperbolic 3-manifolds with some flat
tetrahedra  maximizes the volume in the closure of the space of
non-negative angle structures.  In the case all tetrahedra are
non-flat, this was proved by Casson and Rivin (see \cite{FG},
\cite{Ch} for a written proof). We remark that the corresponding
result also holds for hyper-ideal geodesic triangulations of
compact hyperbolic 3-manifolds with totally geodesic boundary.
This will be addressed in another paper.

\subsection{ } Recall that a triangulated closed pseudo
3-manifold $(M^*, T^*)$ is the quotient of a disjoint union of
tetrahedra so that co-dimension-1 faces are identified in pairs by
affine homeomorphisms. The simplices in the triangulation $T^*$
are the quotients of simplices in the disjoint union. If we remove
all vertices of $T^*$ from $M^*$, the result is an \it ideal
triangulated non-compact 3-manifold $(M, \bold T)$\rm.  We say $M$
has \it cusp ends \rm if the link of each vertex in $T^*$ is a
torus or a Kleinbottle.  We will deal with ideal triangulated cusp
end manifolds  $(M, \bold T)$  in this paper.  An \it angle
structure \rm on $(M, \bold T)$, introduced by Casson, Rivin and
Lackenby (\cite{La}), assigns each edge of each tetrahedron a
positive number, called the angle, so that

(1) the sum of three angles at edges from each vertex of each
tetrahedron is $\pi$, and

(2) the sum of angles around each edge is $2\pi$.

An \it angled tetrahedron \rm is a tetrahedron so that each edge has
assigned a positive number called the angle so that condition (1)
above holds. Given an angled tetrahedron, there is a unique ideal
hyperbolic tetrahedron, call the \it geometric realization\rm,
whose dihedral angles are the assigned angles.  The volume of an
angled tetrahedron is defined to be the volume of its geometric
realization. The volume of an angle structure is the sum of the
volume of its angled tetrahedra.

For an ideal triangulated 3-manifold $(M, \bold T)$ with $n$
tetrahedra, let $\bold A(\bold T) \subset \bold R^{6n}$ be the
space of all angle structures on $\bold T$ and let $vol: \bold
A(\bold T) \to \bold R$ be the volume function. By the
Lobachevsky-Milnor's formula for volume, the volume function
extends continuously to $vol: \overline{\bold A(T)} \to \bold R$
where $\overline{\bold A(\bold T)}$ is the (compact) closure of
$\bold A(\bold T)$ in $\bold R^{6n}$.

\begin{theorem}[Casson-Rivin]
For an ideal triangulated cusped 3-manifold $(M, \bold T)$ so that
$\bold A(\bold T) \neq \emptyset$, a point $p \in \bold A(\bold T)$
corresponds to a complete hyperbolic structure on $M$ if and only
if $p$ is the maximum point of the volume function $vol: \bold
A(\bold T) \to \bold R$.
\end{theorem}

Since the space $\bold A(\bold T)$ is non-compact, there is no guarantee
that the maximum point of $vol$ exists.  Our main theorem
generalizes theorem 1.1 in one direction.

\begin{theorem}
For an ideal triangulated cusp end 3-manifold $(M, \bold T)$ so
that $\overline{\bold A(\bold T)} \neq \emptyset$, if a point $p
\in \overline{\bold A(\bold T)}$ corresponds to a complete
hyperbolic structure on $M$, then $p$ is the maximum point of the
volume function $vol: \overline{\bold A(\bold T)} \to \bold R$.
Furthermore, the point $p$ is the unique maximum volume point in
$\overline{\bold A(\bold T)}$.
\end{theorem}

In \cite{FG2}, Futer and Gueritaud found an example of $(M, \bold T)$ so
that $\bold A(\bold T) \neq \emptyset$ and the maximum point of $vol$ on $
\overline{\bold A(\bold T)}$ does not correspond to a complete
hyperbolic metric. (A similar example was also found by Luo-Zheng
\cite{LZ} by making five 2-3 moves on the standard triangulation of the
figure-8 knot complement). This shows that theorem 1.2
cannot be improved to be a necessary and sufficient condition. We
are curious if the uniqueness of the maximum volume on $\overline{\bold A(\bold T)}$ is the
sufficient condition for the completeness of the metric.

Theorem 1.2 will be proved in \S 2.

\subsection{Acknowledgement} We thank D. Futer, F. Gueritaud and S.
Tillmann for helpful conversations. The work is partially
supported by the NSF.

\section{Proof of Theorem 1.2}

Let us begin with some notations and conventions. Let $\bold
R_{>0} = \{ x \in \bold R | x >0\}$ and $\bold R_{ \geq 0} =\{ x
\in \bold R | x \geq 0\}$. Given a set $X$, let $\bold R^X =\{ f:
X \to \bold R\}$ be the vector space of all functions from $X$ to
$\bold R$. The function $x\ln x: (0, \infty) \to \bold R$ is
extended continuously to $[0, \infty) \to \bold R$ by setting $0
\ln 0=0$. Suppose $(M, \bold T)$ is an ideal triangulated
3-manifold so that $V$, $E$ and $T$ are the sets of all (ideal)
vertices, edges and tetrahedra. Let $I =\{ (e, \sigma) \in E
\times T| $  edge $e$ is adjacent to the tetrahedron $\sigma$\}.
An angle structure is a vector in the space $\bold R^I$ satisfying
a set of linear equations and linear inequalities. If $x \in \bold
R^I$, we use $x_i$ to denote $x(i)$. If $i=(e, \sigma) \in I$, we
use $i>e$ and $i < \sigma$ to indicate the incident relation.
Three distinct $i=(e_1, \sigma), j=(e_2, \sigma)$ and $k=(e_3,
\sigma)$ in $I$ so that $e_1, e_2, e_3$ are three edges from the
same vertex in $\sigma$ will be denoted by $\{i,j,k\} \in \Delta$.
Finally,  we use $(e, \sigma) \sim (e', \sigma)$ to indicate that
$e, e'$ are two opposite edges in the same tetrahedron $\sigma$.
Using these notations, we have

$$ A(\bold T) =\{ x \in \bold R^I | \quad (1), (2), (3)  \quad
\text{hold} \}
$$
where
\begin{enumerate}

\item if $\{i,j,k\} \in \Delta$, $x_i+x_j+x_k=\pi$;

\item if $e \in E$, $\sum_{ i \in I, i>e} x_i =2\pi$;

\item $x_i>0$ for all $i \in I$.

\end{enumerate}

Note that condition $(1)$ implies that $x_i = x_j$ when $ i \sim
j$.  The closure $\overline{\bold A(\bold T)}$ of $\bold A(\bold
T)$ is give by $\{ x \in \bold R_{\geq 0}^I | $ (1) and (2) hold
\}. Theorem 1.2 does not assume $\bold A(\bold T) \neq \emptyset$,
but it assumes $\overline{\bold A(\bold T)} \neq \emptyset$.
Elements in $\overline{\bold A(\bold T)}$ will be called \it
non-negative angle structures. \rm

Suppose $\sigma$ is an angled tetrahedron with three angles $x_1,
x_2, x_3$ at three edges from a vertex. Then the
Lobachevsky-Milnor volume formula says the volume $vol(\sigma)$ of
$\sigma$ is $\Lambda(x_1) + \Lambda(x_2) + \Lambda(x_3)$  where
$\Lambda(t) = -\int_0^t \ln| 2 \sin(u)| du$ is the Lobachevsky
function. The function $\Lambda(t)$ is continuous on $[0,\pi]$. In
particular, the volume function $vol: \overline{\bold A(\bold T)}
\to \bold R$ is give by \begin{equation}  vol(x) = \frac{1}{2}
\sum_{i \in I} \Lambda(x_i). \end{equation} Note that $vol(x)
=\sum_{\sigma \in T} vol(\sigma)$ where $\sigma$ is the angled
tetrahedron with angles given by $x$.

\subsection{ Set up the proof}
Assume that $p \in \overline{\bold A(\bold T)}$ corresponds to the
complete hyperbolic metric on $M$. If $p \in \bold A(\bold T)$, then
Casson-Rivin's theorem implies that $p$ is the maximum point of
the volume.  It remains to deal with $p \in \partial \bold A(\bold T) =
\overline{\bold A(\bold T)} -\bold A(\bold T)$.  Take $q \in \overline{\bold
A(\bold T)}$ so that $q \neq p$. The goal is to show that $vol(p) >
vol(q)$. To this end, let $f(t) = vol( (1-t)p + tq)$ for $t \in
[0, 1]$. We will show that $f(0) > f(1)$, i.e., $vol(p) > vol(q)$.

\begin{lemma}  The function $f(t)$ is concave in $[0,1]$ and is
strictly concave in the open interval $(0, 1)$.
%
\end{lemma}

\begin{proof}  By a result of Rivin \cite{Ri}, the volume function
$\Lambda(t_1) + \Lambda(t_2) + \Lambda(t_3)$ is strictly concave
in the set $\{ (t_1, t_2, t_3) \in \bold R^3 | t_1+t_2+ t_3 = \pi,
$ $t_i >0$ for $i=1,2,3$\}.  In particular, this implies that the
function $vol(x)$ is concave in $x \in \overline{\bold A(\bold T)}$.
Thus $f(t)$ is concave in $[0, 1]$. To see the strictly concavity,
note that $vol(p) >0$ since it is the volume of a complete
hyperbolic structure. In particular, there is tetrahedron $\sigma
\in T$ so that its angles in $p$ are all positive. This implies
that for $t \in (0, 1)$, the angles of $\sigma$ in $(1-t)p+tq$ are
positive. By Rivin's theorem, the volume $\sum_{ i \in I, i <
\sigma} \Lambda((1-t)p_i + tq_i)$ is strictly concave in $t \in
(0, 1)$.  Since $f(t)$ is the sum of concave functions in $t$ so
that one of then is strictly concave, it follows that $f(t)$ is
strictly concave in $(0, 1)$.
\end{proof}

For $t \in (0, 1)$, by the definition of the volume (2.1), we have

\begin{equation} f'(t) =-\frac{1}{2} \sum_{ i \in I} (q_i-p_i) \ln |2 \sin
((1-t)p_i + t q_i)|
\end{equation}
Note that we have used the convention that $0 \ln 0=0$ in (2.2).
Indeed, if $p_i = q_i$ is $0$ or $\pi$, then the term in (2.2)
corresponding to $i$ is defined to be $0$.  (This is due to the
fact that $\Lambda(x) + \Lambda(y) + \Lambda(\pi-x-y) =0$ if $x
\in \pi \bold Z$. )

The goal is to show that
\begin{equation}
\lim_{ t \to 0^+} f'(t)  \leq 0.
\end{equation}
Note that (2.3) and lemma 2.1 imply that $f(0) > f(1)$. In the
rest of the subsections, we will focus on proving (2.3).

\subsection{}
Let $J =\{ i \in I|  p_i =0 $ or $p_i =\pi$\}. Note that if
$(e,\sigma) \in J$, then $(e', \sigma) \in J$ for all other edges
$e'$ in $\sigma$ by the definition of flat tetrahedron, i.e., all
its dihedral angles are $0$ or $\pi$.  Let $a = q-p \in \bold
R^I$.

\begin{lemma}

\begin{enumerate}
\item If $\{i,j,k \} \in \Delta$, then $a_i+a_j+a_k =0$, i.e., for
each tetrahedron $\sigma$, $\frac{1}{2}\sum_{ i< \sigma} a_i =0$.

\item  For each edge $e \in E$, $\sum_{ i > e} a_i =0$.

\item   $\sum_{i \in I} a_i = 0$.

\end{enumerate}
\end{lemma}
Indeed, the first two conditions follows from the definition of
angle structures $(1)$ and $(2)$. The last condition follows
from part (1) by summing over all tetrahedra and then divided by
2.

By lemma 2.2(3), we can rewrite $f'(t)$ in (2.2) as
\begin{equation}
f'(t) = - \frac{1}{2}\sum_{i \in I} a_i \ln | \sin( (1-t) p_i + t
q_i)|
\end{equation}

 The following was proved in \cite{Lu1}. It can also be found in \cite{FG2}.

\begin{lemma}
\begin{equation}
\lim_{t \to 0^+} f'(t) = -\frac{1}{2}(\sum_{i \notin J} a_i \ln
|\sin(p_i)| - \sum_{i \in J} a_i \ln | a_i|).
\end{equation}
\end{lemma}

\subsection{Penner's decorated ideal simplexes}

To understand the right-hand-side of (2.5), we need a proposition
about the geometry of decorated ideal hyperbolic tetrahedra.
Following Penner \cite{Pe}, a \it decorated ideal n-simplex \rm is
an ideal hyperbolic n-simplex so that each vertex is assigned a
horosphere centered at the vertex.  If $\sigma$ is a decorated
ideal n-simplex and $e$ is an edge of it, the length $L(e)$ of $e$
is defined to be the signed distance between the two horospheres
centered at the end points of $e$ (the distance is negative if the
horospheres intersect). More precisely, suppose $p,p'$ are the two
points of intersection of $e$ with these two horocspheres. Then
$L(e)$ is $dist(p, p')$ if these two horospheres are disjoint and
is $-dist(p, p')$ if they intersect.

\begin{lemma}
Suppose $\sigma$ is a decorated ideal hyperbolic tetrahedron with
edge length $L(e)$ and dihedral angle $\theta(e)$ at the edge $e$.
Assume that $\theta(e) \in (0, \pi)$ for all edges. Then there is
a constant $c(\sigma)$ depending only on $\sigma$ so that for any
pairs of opposite edges $e, e'$ in $\sigma$,

\begin{equation}
\frac{1}{2} (L(e) + L(e')) = \ln |\sin (\theta(e))| + c(\sigma)
\end{equation}
\end{lemma}

\begin{proof}
The proof is based on the cosine law for decorated ideal triangles
first discovered by Penner \cite{Pe}. Namely, give a decorated
ideal triangle of lengths $l_1, l_2, l_3$, the "angles" of the
triangle, denoted by $a_1, a_2, a_3$, are the lengths of the
portion of the horocycle inside the triangle. Indices are arranged
so that the angle $a_i$ is facing the edge of length $l_i$.  The
the cosine law says
\begin{equation}
l_i = -(\ln a_j + \ln a_k) \quad  \{i,j,k\} =\{1,2,3\}.
\end{equation}

For the edge $e$  (respectively $e'$), there are two face
triangles of the tetrahedron $\sigma$ having $e$ (resp. $e'$) as
an edge. These face triangles are naturally decorated ideal
hyperbolic triangles. Let $a_1, a_2, a_3, a_4$ (resp. $a_1'$,
$a_2'$, $a_3'$, $a_4'$) be the inner angles of these decorated
face ideal triangles so that $a_i$'s are adjacent to $e$ (resp.
$e'$). Let the rest of the four face angles (of the four decorated
ideal triangles) be $b_1, ..., b_4$. Here vertices of angles $b_i$
are either in $e$ or $e'$. Then by the cosine law, we have

\begin{equation}
L(e) = -\frac{1}{2} \sum_{i=1}^4 \ln a_i  \quad \text{and} \quad
L(e') =-\frac{1}{2} \sum_{i=1}^4 \ln a_i'
\end{equation}

This shows
$$ L(e) + L(e') = -\frac{1}{2} \sum_{i=1}^4 (\ln a_i + \ln a_i')
$$
$$=\frac{1}{2}c_0(\sigma) + \frac{1}{2} \sum_{i=1}^4 \ln b_i$$
where $c_0(\sigma) =  \sum_{i=1}^4 (\ln a_i + \ln a_i' + \ln b_i)$
is the sum over all twelve face angles.

Consider the Euclidean triangles obtained by intersecting the
horospheres with the ideal tetrahedron. The dihedral angles
$\theta(e)$'s are the inner angles and the face angles $b_i$'s are
the edge lengths of the Euclidean triangles.  Thus, by the Sine
law for Euclidean triangles, we can write
$$ \ln b_i = c_i(\sigma) + \ln |\sin (\theta(e)|$$
where $b_i$ has its vertex at $e$.  Putting these together and
using the fact that $\theta(e) = \theta(e')$, we obtain (2.6)
where $c(\sigma) =\frac{1}{2}\sum_{i=0}^4 c_i(\sigma)$.
\end{proof}

For a decorated ideal triangle of edge lengths $L(e)$, we define
the \it average edge length \rm of $e$ to be $W(e)
=\frac{1}{2}(L(e) + L(e'))$ where $e, e'$ are opposite edges.

\begin{lemma}
For a decorated ideal tetrahedron $\sigma$, if $e_1, e_2, e_3$ are
three edges from a vertex $v$, then
\begin{equation}
e^{W(e_1)} + e^{W(e_2)} \geq e^{W(e_3)}
\end{equation}
so that equality holds if and only if $\theta(e_3)=\pi$,
$\theta(e_1) =\theta(e_2) =0$.
\end{lemma}

Indeed, consider the Euclidean triangle obtained by intersecting
the horosphere centered at the vertex $v$ with the ideal
tetrahedron. The inner angles of the Euclidean triangle are
$\theta(e_i)$'s and the edge lengths of it are $R \sin(\theta
(e_i))$ where $R$ is the radius of the circumcircle.  Now by lemma
2.4 that $\sin(\theta(e_i)) = c' e^{W(e_i)}$, the lengths of the
edges in the triangle are  $c e^{W(e_i)}$ for some constant $c$.
Thus the lemma follows from the triangular inequality for edge
lengths of triangles.

\subsection{A proof of theorem 1.2}
Recall that  the maximum volume point $p \in \overline{\bold
A(\bold T)}$ corresponds to a complete hyperbolic metric, i.e.,
there exists a geodesic triangulation of a complete hyperbolic
metric on $M$ so that the triangulation is isotopic to $T$ and the
dihedral angles coincide with the angles given by $p$. Choose
small horospheres at the cusp ends of $M$ so that each tetrahedron
becomes an ideal decorated hyperbolic tetrahedron.  In particular,
each edge $e$ in $\bold T$ has a the edge length $L(e)$ (in the
decorated tetrahedra). For each $i=(e, \sigma) \in I$, we define
the average length $w_i$ (of $e$ in $\sigma$) to be
\begin{equation} w_i =\frac{1}{2} (L(e) + L(e'))
\end{equation} where $e, e'$ are opposite edges in $\sigma$.

\begin{lemma}  We have
\begin{equation}
\sum_{ i \notin J} a_i \ln |\sin(p_i)| =\sum_{ i \notin J} a_i
w_i.
\end{equation}
\end{lemma}

\begin{proof}  A tetrahedron  $\sigma$ is called flat (in $p$) if its dihedral angles in $p$ are $0, 0, \pi$, i.e., there is
$i \in I$ with $i < \sigma$. If $\sigma$ is not flat, then by
lemma 2.4, there is a constant $c(\sigma)$ so that (2.6) holds for
each pair of opposite edges $e, e'$.  This is the same as
\begin{equation}
w_i = \ln | \sin(p_i)| + c(\sigma)
\end{equation} for $i < \sigma$.
 Multiply (2.12) by $a_i$  and sum over all not flat tetrahedra, we
 obtain
\begin{equation}
\sum_{ i \notin J} a_i w_i =  \sum_{ i \notin J} a_i \ln
|\sin(p_i)| + \sum_{i \notin J, i < \sigma }  a_i c(\sigma).
\end{equation}

But
$$ \sum_{i \notin J, i < \sigma }  a_i c(\sigma)
= \sum_{ \sigma {\text{ not flat}}}  c(\sigma)( \sum_{ i < \sigma}
a_i )=0$$ due to lemma 2.2 (1).  This ends the proof.
\end{proof}

On the other hand, we have
\begin{equation}
\sum_{ i \in I}  a_i w_i =0.
\end{equation}
Indeed,  if $i=(e, \sigma)$ and $j=(e', \sigma)$ where $e, e'$ are
opposite edges in $\sigma$, then $a_i = a_j$ and $w_i = w_j$.
Furthermore by (2.10), $a_i w_i + a_j w_j = a_i L(e) + a_j L(e')$.
Thus
$$ \sum_{ i \in I} a_i w_i = \sum_{i \in I} a_i ( \sum_{ e < i}
L(e))$$
$$=\sum_{ e \in E} L(e) (\sum_{ i < e} a_i)$$
$$=0$$ due to
lemma 2.2 (2).

By combining (2.11) and (2.14), we obtain
\begin{equation}
\sum_{ i \notin J} a_i \ln |\sin(p_i)|=- \sum_{ i \in J} a_i w_i.
\end{equation}

By (2.15), we can rewrite (2.5) as,
$$\lim_{t \to
0^+} f'(t) = -\frac{1}{2}(\sum_{ i \in J} a_i \ln |a_i| + \sum_{ i
\in J} a_i w_i)$$
$$
=\sum_{ \sigma \text{  is flat}} (-a_i \ln |a_i| -a_j \ln |a_j| -a_k
\ln |a_k| + a_iw_i+a_j w_j +a_k w_k)
$$
where $i,j,k < \sigma$, $\{i,j,k\} \in \Delta$.

Since $\sigma$ is flat, we may assume that $p_k=\pi$, $p_i=p_j=0$.
Then by lemma 2.5, three average lengths $w_i, w_j, w_k$ satisfy
the triangular equality, i.e., $e^{w_k} = e^{w_i} + e^{w_j}$.
Furthermore, $a_i \geq 0$, $ a_j \geq 0$ and $a_i+a_j+a_k=0$.

We claim  \begin{equation}
-a_i \ln |a_i| -a_j \ln |a_j| -a_k \ln |a_k| + a_iw_i+a_j
w_j +a_k w_k \leq 0.
\end{equation}

Evidently, (2.16) implies that $\lim_{t \to 0^+} f'(t) \leq 0$.
Now (2.16) follows from the following simple lemma on a convex
function where we take $x=a_i, y=a_j, z=a_k, a=w_i, b=w_j, c=w_k$.

\begin{lemma} Suppose $x,y,a,b,c \in \bold R_{\geq 0}$,  $x+y+z=0$ and $e^c \geq e^a+e^b$. Then
\begin{equation}
-x\ln x -y\ln y -z \ln z +  ax+by+cz \leq 0.
\end{equation}
\end{lemma}

We remark that if $e^c = e^a+e^b$, then the inequality becomes
equality for some non-zero $x, y, z$.

\begin{proof} Replacing $z=-x-y$, we obtain the equivalent form of (2.17) as
$$(x+y) \ln (x+y) - x \ln x - y \ln y \leq (c-a) x + (c-b) y.$$

The above inequality is homogeneous in $(x,y)$, i.e., it is
equivalent if we replace $(x,y)$ by $(\lambda x, \lambda y)$ where
$\lambda >0$. Thus we may assume further that $y=1$. Thus it
remains to prove,
\begin{equation}
(x+1)\ln(x+1) -x \ln x \leq (c-a) x + (c-b)
\end{equation} for all $x \geq 0$.
Let $g(x) =(x+1) \ln (x+1) -x\ln x$. Then $g'(x) = \ln( 1+
\frac{1}{x})$ and $g''(x) =-\frac{1}{x(1+x)}$. It follows that
$g''(x) \leq 0$, i.e., $g$ is concave in $[0, \infty)$. The
equation of the tangent line to $g$ at the point
$x_0=\frac{e^a}{e^c-e^a}$ is $y=h(x)$ where
$$  h(x) = (c-a) x + (c - \ln (e^c -e^a)).$$
Since $g$ is concave, we have $g(x) \leq h(x)$. Now use $b \geq
\ln (e^c-e^a)$, we obtain $h(x) \leq (c-a)x + (c-b)$. Thus lemma
follows.
\end{proof}

\bibliographystyle{amsalpha}

\end{document}